\documentclass[oneside,english]{amsart}
\usepackage[T1]{fontenc}
\usepackage[latin9]{inputenc}
\usepackage{amstext}
\usepackage{amsthm}
\usepackage{amssymb}
\usepackage{float}
\usepackage{ifthen}

\makeatletter
\numberwithin{equation}{section}
\numberwithin{figure}{section}
  \theoremstyle{plain}
  \newtheorem*{question*}{\protect\questionname}
\theoremstyle{plain}
\newtheorem{thm}{\protect\theoremname}[section]
  \theoremstyle{plain}
  \newtheorem{conjecture}[thm]{\protect\conjecturename}
  \theoremstyle{plain}
  \newtheorem{lem}[thm]{\protect\lemmaname}

\allowdisplaybreaks
\usepackage{tikz, pgf}
\usepackage{hyphenat}

\makeatother

\usepackage{babel}
  \providecommand{\conjecturename}{Conjecture}
  \providecommand{\lemmaname}{Lemma}
  \providecommand{\questionname}{Question}
\providecommand{\theoremname}{Theorem}

\begin{document}
\def\paperstyle{ams}
\newcommand{\customizepaperstyle}[2]{%
	\ifthenelse{\equal{\paperstyle}{#1}}{#2}{}%
}%
\newcommand{\Marten}{Marten Wortel}
\newcommand{\Miek}{Miek Messerschmidt}

\newcommand{\MartenEmail}{marten.wortel@gmail.com}
\newcommand{\MiekEmail}{mmesserschmidt@gmail.com}

\newcommand{\NWUAddress}{Unit for BMI; North-West University; Private Bag X6001; Potchefstroom; South Africa; 2520}
\newcommand{\NWUAddressLineBreaks}{%
Unit for BMI\\
North-West University\\
Private Bag X6001\\
Potchefstroom\\
South Africa\\
2520}

\newcommand{\COEDisclaimer}{%
Both authors were partially funded by DST-NRF Centre of Excellence in Mathematical and Statistical Sciences (CoE-MaSS). Opinions expressed and conclusions arrived at are those of the authors and are not necessarily to be attributed to the CoE-MaSS}

\newcommand{\MSECodesPrimary}{%
	46B20
}

\newcommand{\MSECodesSecondary}{%
	51F99, 
	46B07
}

\newcommand{\paperkeywords}{Intrinsic metric; unit sphere; normed space; local theory}

\global\long\def\norm#1{\left\Vert #1\right\Vert }

\global\long\def\norminline#1{\|#1\|}

\global\long\def\abs#1{\left|#1\right|}

\global\long\def\set#1#2{\left\{  \vphantom{#1}\vphantom{#2}#1\right.\left|\ #2\vphantom{#1}\vphantom{#2}\right\}  }

\global\long\def\sphere#1{\mathbf{S}_{#1}}

\global\long\def\closedball#1{\mathbf{B}_{#1}}

\global\long\def\openball#1{\mathbb{B}_{#1}}

\global\long\def\duality#1#2{\left\langle \vphantom{#1}\vphantom{#2}#1\right.\left|\ #2\vphantom{#2}\right\rangle }

\global\long\def\pathspace#1#2{\mathcal{P}_{#1}(#2)}

\global\long\def\planarpathspace#1#2{\mathcal{P}_{#1}^{\textup{planar}}(#2)}

\global\long\def\lengthoperator#1{L_{#1}}

\global\long\def\pathmetric#1#2{d_{#1,#2}}

\global\long\def\planarpathmetric#1#2{d_{#1,#2}^{\textup{planar}}}

\global\long\def\ray#1#2{r_{#1,#2}}

\global\long\def\parenth#1{\left(#1\right)}

\global\long\def\curly#1{\left\{  #1\right\}  }

\global\long\def\span{\textup{span}}

\global\long\def\image{\textup{Im}}

\global\long\def\R{\mathbb{R}}

\global\long\def\N{\mathbb{N}}

\global\long\def\Rnonneg{\mathbb{R}_{\geq0}}

\title[The intrinsic metric on unit spheres]{The intrinsic metric on the \\
unit sphere of a normed space}

\author{\Miek}
\author{\Marten}

\address{\Miek; \NWUAddress}
\email{\MiekEmail}

\address{\Marten; \NWUAddress}
\email{\MartenEmail}

\thanks{\COEDisclaimer}

\subjclass[2000]{\MSECodesPrimary\ (primary). \MSECodesSecondary\ (secondary)}

\begin{abstract}
Let $S$ denote the unit sphere of a real normed space. We show that
the intrinsic metric on $S$ is strongly equivalent to the induced
metric on $S$. Specifically, for all $x,y\in S$, 
\[
\norm{x-y}\leq d(x,y)\leq\sqrt{2}\pi\norm{x-y},
\]
where $d$ denotes the intrinsic metric on $S$. 
\end{abstract}

\maketitle


\section*{Correction}
The authors would like to thank the anonymous referee who pointed out the result
\cite[Theorem~3.5]{Schaffer} to them. This result solves the problem posed in this manuscript
and far predates it, (it actually proves Conjecture~\ref{conj} stated below).

The authors take some solace in the following:
That the result is true, with optimal constant $2$, as conjectured.
That their initial \emph{this-really-should-exist-somewhere-in-the-literature} gut feeling was actually correct.
And finally, that this result is now much easier to find by googling for the obvious ``intrinsic metric unit sphere''.

\section{Introduction}

Consider the following problem which arose in other questions under
investigation by the first author:
\begin{question*}
For the unit sphere of a real normed space, is the induced metric
strongly equivalent to the sphere's intrinsic metric?
\end{question*}
This paper will answer this question in the affirmative.

More precisely: For a unit sphere $S$ of a real normed space, the
length of a path $\rho:[0,1]\to S$ (assumed to be continuous), is
given by
\[
L(\rho):=\sup\set{\sum_{j=0}^{n-1}\norm{\rho(t_{j})-\rho(t_{j+1})}}{\begin{array}{c}
n\in\mathbb{N},\\
0=t_{0}<\ldots<t_{n}=1
\end{array}},
\]
and the intrinsic metric on $S$ is defined by taking the infimum
of the above quantity over all paths between two points. I.e., for
$x,y\in S$, we define the \emph{intrinsic metric on $S$} by
\[
d(x,y):=\inf\set{L(\rho)}{\begin{array}{c}
\rho:[0,1]\to S\ \mbox{continuous},\\
\rho(0)=x,\ \rho(1)=y
\end{array}}.
\]
The above question is then rephrased as: For the unit sphere $S$
in a normed space, do there exist constants $A,B>0$ such that, for
all $x,y\in S$,
\[
A\norm{x-y}\leq d(x,y)\leq B\norm{x-y}?
\]
To the authors' knowledge, the answer to this question does not appear
in the literature\footnote{False! See the correction above!},
a fact that is perhaps more surprising than the
positive solution to the problem which will be presented in this paper.

The problem can essentially be reduced to one in two dimensions and
apart from relying on John Ellipsoids---a fundamental structure from
local theory---the result follows from entirely elementary (albeit
somewhat technical) arguments.

We will now describe the structure of the paper.

After introducing the needed notation, definitions and preliminary
results in Section~\ref{sec:Preliminaries}, in Section~\ref{sec:R2-results}
our goal will be to prove Theorem~\ref{thm:main-result2}:

\setcounter{thm}{5}
\setcounter{section}{3}

\begin{thm}
For any norm $\norm{\cdot}$ on a real vector space $V$, let $d$
denote the intrinsic metric on the unit sphere $S$ of $V$. For all
$x,y\in S$,
\[
\norm{x-y}\leq d(x,y)\leq\sqrt{2}\pi\norm{x-y}.
\]

\end{thm}
\setcounter{thm}{0}
\setcounter{section}{1}

A crucial ingredient is that of the John Ellipsoid: The largest ellipsoid
(Euclidean ball of largest volume) that can be contained inside a
unit ball of a finite dimensional normed space.
\def\johncite{\customizepaperstyle{lms}{(}John's Theorem \cite[Theorem 12.1.4]{AlbiacKalton}\customizepaperstyle{lms}{)}}%
\begin{thm}[\johncite]
\label{thm:John's-Theorem}
Let $W$ be any normed space of dimension $n>1$. With $\norm{\cdot}_{E}$
denoting the Euclidean norm on $\R^{n}$, there exists a norm one
isomorphism $T:(\R^{n},\norm{\cdot}_{E})\to W$ with inverse $T^{-1}:W\to(\R^{n},\norm{\cdot}_{E})$
whose norm is at most $\sqrt{n}$.
\end{thm}
Specifically, all two-dimensional subspaces of a normed space has
a Banach-Mazur distance of at most $\sqrt{2}$ from two-dimensional
Euclidean space. Of course, the intrinsic metric on a normed space
$V$'s unit sphere is bounded above by the ``planar'' intrinsic
metric: where all paths in the defining infimum are taken to live
in any two-dimensional subspace of $V$. This allows us to reduce
the question to $\R^{2}$, where $S$ lies between a Euclidian unit
sphere $\sphere E$ and $\sqrt{2}\sphere E$.

In this setting, the crucial ingredients are Lemmas~\ref{lem:estimate_segment}
and~\ref{lem:norm_decreasing}, which allows us to conclude the local
bi\hyp{}Lipschitzness of the map $\sigma:S\to\sphere E$ defined
by $\sigma(x):=x/\norm x_{E}$ where both $S$ and $\sphere E$ are
endowed with the Euclidean induced metric. Since the Euclidean induced--
and intrinsic metrics on $\sphere E$ are easily calculated and related
(Lemma~\ref{lem:bound-eclidean-arclength-by-shortcut}), our main
results, Theorems~\ref{thm:main-result1} and~\ref{thm:main-result2},
then easily follow.

\medskip

We note that the constant $\sqrt{2}\pi$ obtained in our main result
(Theorem~\ref{thm:main-result2}) is likely not optimal. In two dimensions,
the $\norm{\cdot}_{\infty}$-norm provides a worst case for the John
Ellipsoid. Let $S_{\infty}$ be the unit sphere of this norm, with
intrinsic metric $d_{\infty}$, then for $x,y\in S_{\infty}$, it
is easily seen that $\norm{x-y}_{\infty}\leq d_{\infty}(x,y)\leq2\norm{x-y}_{\infty}$.
This prompts the following conjecture:
\begin{conjecture}\label{conj}\footnote{This conjecture is proven true in \cite[Theorem~3.5]{Schaffer}.}
For any norm $\norm{\cdot}$ on a real vector space $V$, let $d$
denote the intrinsic metric on the unit sphere $S$ of $V$. For all
$x,y\in S$,
\[
\norm{x-y}\leq d(x,y)\leq2\norm{x-y}.
\]

\end{conjecture}

\section{Definitions, notation and preliminary results\label{sec:Preliminaries}}

This section will explicitly define all notation used in this paper.
Since we will translate between many different metrics on many different
sets, we will take extreme care to make our notation as explicit as
possible.

Let $V$ be a real vector space. Let $A$ be an arbitrary index symbol
and $\norm{\cdot}_{A}$ any norm on $V$. We will denote unit sphere,
closed unit ball and open unit ball with respect to $\norm{\cdot}_{A}$
respectively by $\sphere A$, $\closedball A$ and $\openball A.$

For any subset $M\subseteq V$, we define the \emph{induced metric}
$d_{A}:M\to\Rnonneg$ on $M$ by $d_{A}(x,y):=\norm{x-y}_{A}$ for
$v,w\in M$.

We define the \emph{$A$-$M$-path-space} by
\[
\mathcal{P}_{A}(M):=\set{\rho}{\rho:[0,1]\to M,\ \norm{\cdot}_{A}-\mbox{continuous}},
\]
and the \emph{planar $A$-$M$-path-space} by
\[
\planarpathspace AM:=\set{\rho\in\pathspace AM}{\dim\parenth{\span\parenth{\image\rho}}=2}.
\]
We define the \emph{$A$-path-length operator} $L_{A}:\mathcal{P}_{A}(V)\to\mathbb{R}_{\geq0}\cup\{\infty\}$
by
\[
\lengthoperator A(\rho):=\sup\set{\sum_{j=0}^{n-1}\norm{\rho(t_{j})-\rho(t_{j+1})}_{A}}{\begin{array}{c}
n\in\mathbb{N},\\
0=t_{0}<\ldots<t_{n}=1
\end{array}}\quad(\rho\in\mathcal{P}_{A}(V)).
\]
We define the \emph{(extended) $A$-$M$-intrinsic metric }$\pathmetric AM:M\to\Rnonneg\cup\{\infty\}$,
by
\[
\pathmetric AM(v,w):=\inf\set{\lengthoperator A(\rho)}{\begin{array}{c}
\rho\in\pathspace AM,\ \rho(0)=v,\ \rho(1)=w\end{array}}\quad(v,w\in M)
\]
and the \emph{(extended) $A$-$M$-planar-intrinsic metric }$\planarpathmetric AM:M\to\Rnonneg\cup\{\infty\}$
by
\[
\planarpathmetric AM(v,w):=\inf\set{\lengthoperator A(\rho)}{\begin{array}{c}
\rho\in\planarpathspace AM,\ \rho(0)=v,\ \rho(1)=w\end{array}}\quad(v,w\in M).
\]

We introduce the following abbreviated notation that will aid in readability
of the paper: For (extended) metrics $d$ and $d'$ on some set $D$,
subset $M\subseteq D$, and constant $K>0$, by
\[
\mbox{``}d\leq Kd'\quad\mbox{on}\quad M\mbox{''}
\]
we will mean $d(a,b)\leq Kd'(a,b)$ for all $a,b\in M$.

For any real normed space $(V,\norm{\cdot}_{A})$ and subset $M\subseteq V$,
since $\planarpathspace AM\subseteq\pathspace AM\subseteq\pathspace AV$,
the chain of inequalities
\[
d_{A}=\pathmetric AV\leq\pathmetric AM\leq\planarpathmetric AM\leq\infty\quad\mbox{on}\quad M
\]
is easy to verify.

An elementary calculation establishes the following lemma:
\begin{lem}
\label{lem:bound-eclidean-arclength-by-shortcut}With $\norm{\cdot}_{E}$
denoting the Euclidean norm on $\mathbb{R}^{2}$,
\[
d_{E}\leq\pathmetric E{\sphere E}\leq\frac{\pi}{2}d_{E}\quad\mbox{on}\quad\sphere E.
\]

\end{lem}
If $\R^{2}$ is endowed with the euclidean norm $\norm{\cdot}_{E}$
arising from an inner product $\duality{\cdot}{\cdot}$, for elements
$x,y\in\R^{2}$ the \emph{ray from $x$ through $y$ }is denoted by
$\ray xy$ and defined by $\ray xy:=\set{(1-t)x+ty}{t\geq0}$. For a point
$x$ and points $y,z\in\R^{2}$ distinct from $x$, when referring to the
size of the angle between $\ray xy$ and $\ray xz$ we will mean the
quantity
\[
\arccos\parenth{\frac{\duality{y-x}{z-x}}{\norm{y-x}_{E}\norm{z-x}_{E}}}\in[0,\pi].
\]
For points $v,w,x,y\in\R^{2}$, we will say \emph{the ray
$\ray xy$ lies between the rays $\ray xv$ and $\ray xw$ } if $v,w$ and $x$
are in general position and $\ray xy\cap\set{(1-t)v+tw}{t\in[0,1]}\neq\emptyset$, or,
$\ray xy =\ray xv = \ray xw$.

\section{The intrinsic metric on unit spheres in $\protect\R^{2}$\label{sec:R2-results}}

In this section we will prove our main results. Although somewhat
technical, our results follow mostly from elementary trigonometry and Euclidian
plane geometry.

Let $\norm{\cdot}_{E}$ denote the Euclidean norm on $\R^{2}$ and
let $\norm{\cdot}_{X}$ be any norm on $\R^{2}$ satisfying $\closedball E\subseteq\closedball X\subseteq K\closedball E$
for some $K\geq1$. A large part of our attention will be devoted
to proving that the map $\sigma:\sphere X\to\sphere E$ defined by
$\sigma(x):=x/\norm x_{E}$ is locally bi\hyp{}Lipschitz when both
$\sphere X$ and $\sphere E$ are both endowed with the Euclidean
induced metrics. Once this has been achieved through Lemmas~\ref{lem:estimate_segment}
and~\ref{lem:norm_decreasing}, a straightforward calculation will
prove our main results Theorems~\ref{thm:main-result1} and~\ref{thm:main-result2}.

\medskip

Let $\norm{\cdot}_{E}$ denote the Euclidean norm on $\R^{2}$ and
let $\norm{\cdot}_{X}$ be any other norm on $\R^{2}$ satisfying
$\closedball E\subseteq\closedball X$. We will first relate points
on $\sphere X$ to lines tangent to $\sphere E$. Specifically, for
any point $x\in\sphere X$ that is not in $\sphere E$, the two lines
through $x$ that are tangent to $\sphere E$ are such that points
in $\sphere X$ ``close to'' $x$ are ``wedged between'' the tangent
lines. Also, if $x\in\sphere X\cap\sphere E$, then the whole of $\sphere X$
lies on the same side of the line $\set{y\in\R^{2}}{\duality xy=1}.$

Let $x\in\R^{2}\setminus\openball E$, and let $\tau(x)\in\sphere E$
be a point on a tangent line to $\sphere X$ through $x$. Then the
angle between $r_{0,x}$ and $r_{0,\tau(x)}$ equals $\arccos(\norm x_{E}^{-1})$.
Let $x^{\perp}\in\sphere X\cap\{x\}^{\perp}$ be such that $\duality{\tau(x)}{x^{\bot}}\geq0$.
If we now define
\[
a(x):=\cos\left(\arccos\left(\frac{1}{\norm x_{E}}\right)\right)\frac{x}{\norm x_{E}}=\frac{x}{\norm x_{E}^{2}}
\]
and
\[
b(x):=\sin\left(\arccos\left(\frac{1}{\norm x_{E}}\right)\right)x^{\perp}=\sqrt{1-\frac{1}{\norm x_{E}^{2}}}x^{\perp},
\]
then $\tau(x)=a(x)+b(x)$.
\begin{lem}
\label{lem:tangent-lines} Let $\duality{\cdot}{\cdot}$ be an inner
product and $\norm{\cdot}_{E}$ be the associated Euclidean norm on
$\R^{2}$. Let $\norm{\cdot}_{X}$ be any other norm on $\R^{2}$
such that $\closedball E\subseteq\closedball X$. For all $x\in\sphere X$,
\begin{enumerate}
\item For all $t\in\R$, $\duality{tx+(1-t)\tau(x)}{\tau(x)}=1.$
\item For all $t\in[0,1]$, $\norm{tx+(1-t)\tau(x)}_{X}\leq1$.
\item For all $t>1$, $\norm{tx+(1-t)\tau(x)}_{X}\geq1$.
\item If $x\in\sphere X\cap\sphere E$ and $y\in\sphere X$, then $\duality xy\leq1$.
\end{enumerate}
\end{lem}
\begin{proof}
We prove (1). Let $x\in\sphere X$. For all $t\in\R$,
\begin{eqnarray*}
\duality{tx+(1-t)\tau(x)}{\tau(x)} & = & t\duality x{\tau(x)}+(1-t)\duality{\tau(x)}{\tau(x)}\\
 & = & t\frac{\duality xx}{\norm x_{E}^{2}}+(1-t)\\
 & = & t+(1-t)\\
 & = & 1.
\end{eqnarray*}

We prove (2). Let $x\in\sphere X$. Since $\tau(x)\in\sphere E\subseteq\closedball X$,
and $\closedball X$ is convex, the result follows.

We prove (3). Let $x\in\sphere X$. Since $\norm x_{X}=1$, if $\tau(x)=x$,
then, $\norminline{tx+(1-t)\tau(x)}_{X}=1$, and the result is trivial.
We therefore assume $\tau(x)\neq x$. Since $\tau(x)\in\sphere E$,
so that$\norm{\tau(x)}_{X}\leq1$, by the reverse triangle inequality
and intermediate value theorem there exists some $t_{0}\leq0$ such
that $1=\norminline{t_{0}x+(1-t_{0})\tau(x)}_{X}=\norminline{1x+(1-1)\tau(x)}_{X}$
(here we used $\tau(x)\neq x$). Since the map $t\mapsto\norminline{tx+(1-t)\tau(x)}_{X}$
is convex, we cannot have that $\norminline{tx+(1-t)\tau(x)}_{X}<1$
for any $t>1$, as this would contradict $1=\norminline{t_{0}x+(1-t_{0})\tau(x)}_{X}=\norminline{1x+(1-1)\tau(x)}_{X}$.
We conclude that $\norminline{tx+(1-t)\tau(x)}_{X}\geq1$ for all
$t>1$.

We prove (4). Let $x\in\sphere X\cap\sphere E$ and $y\in\sphere X$,
but suppose $\duality xy>1$.

If $y$ and $x$ are linearly dependent, then $\norm y_{X}>1$, contradicting
$y\in\sphere X$, and we therefore may assume that $y$ and $x$ are
linearly independent.

Let $L$ denote the line through $x$ and $y$, parameterized by the
affine map $\eta(t):=(1-t)y+tx$ for $t\in\R$. The line $L$ is not
tangent to $\sphere E$ (else we would have $\duality xy=1$). Therefore
$L$ intersects $\sphere E$ in two distinct points, one being $x$;
let $t_{0}\in\R$ be such that $\eta(t_{0})\in\sphere E\cap L$ is
the other. We must have $t_{0}>1$, since $1<\duality x{\eta(t)}$
for $t\in[0,1)$. Since $\eta$ is an affine map and $(\R^{2},\norm{\cdot}_{E})$
is a strictly convex space,
\begin{eqnarray*}
\norm{\eta\parenth{\frac{1+t_{0}}{2}}}_{X} & = & \norm{\frac{\eta(t_{0})+\eta(1)}{2}}_{X}\\
 & \leq & \norm{\frac{\eta(t_{0})+\eta(1)}{2}}_{E}\\
 & < & 1.
\end{eqnarray*}

Let $\lambda:=2(1+t_{0})^{-1}\in(0,1)$, so that $\lambda\parenth{\frac{1+t_{0}}{2}}+(1-\lambda)0=1$.
Then, again since $\eta$ is affine,
\begin{eqnarray*}
1 & = & \norm x_{X}\\
 & = & \norm{\eta(1)}_{X}\\
 & = & \norm{\eta\parenth{\lambda\parenth{\frac{1+t_{0}}{2}}+(1-\lambda)0}}_{X}\\
 & = & \norm{\lambda\eta\parenth{\frac{1+t_{0}}{2}}+(1-\lambda)\eta(0)}_{X}\\
 & \leq & \lambda\norm{\eta\parenth{\frac{1+t_{0}}{2}}}_{X}+(1-\lambda)\norm{\eta(0)}_{X}\\
 & = & \lambda\norm{\eta\parenth{\frac{1+t_{0}}{2}}}_{X}+(1-\lambda)\norm y_{X}\\
 & < & \lambda\cdot1+(1-\lambda)\cdot1\\
 & = & 1,
\end{eqnarray*}
which is absurd. We conclude that $\duality xy\leq1$ for all $x\in\sphere X\cap\sphere E$
and all ${y\in\sphere X}$.
\end{proof}
Next, we show that for points $x,y\in\sphere X$ that are ``sufficiently
close'', the size of the angle formed by the rays $\ray 0x$ and $\ray 0y$
bounds the size of the acute angle formed by the ray $\ray xy$ and
the perpendicular line to $\ray 0x$ through $x$.

\begin{lem}
\label{cor:angles} Let $\norm{\cdot}_{E}$ denote the Euclidean norm
on $\R^{2}$ and $\norm{\cdot}_{X}$ be any norm on $\R^{2}$ such
that $\closedball E\subseteq\closedball X$. Let $x,y\in\sphere X$
and let $x^{\bot}\in\sphere E\cap\{x\}^{\bot}$ be such that $\duality{x^{\bot}}y\geq0$
and define $v:=x+x^{\bot}$. If $K\geq1$ and $x,y\in\sphere X$ are
such that $\norm x_{E}\leq K$ and the size of the angle between the
rays $\ray 0x$ and $\ray 0y$ is at most $\arccos(K^{-1})$, then
$\alpha$, the size of the angle between the rays $\ray xv$ and $\ray xy$,
is also at most $\arccos(K^{-1})$.\end{lem}
\begin{proof}
As a visual aid, the reader is referred to Figure~\ref{figure:angles}.

\customizepaperstyle{lms}{%
	\begin{figure}[t]
		\centering
\definecolor{uququq}{rgb}{0,0,0}
\begin{tikzpicture}[line cap=round,line join=round,x=3.0cm,y=3.0cm]
\clip(-1.1,-0.2) rectangle (2.7,2.1);
\fill[line width=0pt,fill=black,fill opacity=0.25] (1.74,3.44) -- (0,1.5) -- (0.75,0.67) -- (0.83,0.56) -- (0.89,0.46) -- (0.93,0.36) -- (0.98,0.2) -- (4.11,0.83) -- cycle;
\draw(0,0) circle (1);
\draw (0,1.5)-- (0,0);
\draw [dash pattern=on 2pt off 2pt,domain=-0.74535599249993:4.499989389412392] plot(\x,{(--1.12--0.83*\x)/0.75});
\draw [dash pattern=on 2pt off 2pt,domain=-3.982478398090284:0.7453559924999299] plot(\x,{(-1.12--0.83*\x)/-0.75});
\draw [rotate around={-35.33:(0,0)},dash pattern=on 2pt off 5pt] (0,0) ellipse (2.55 and 1.3);
\draw [domain=-3.98:4.5] plot(\x,{(--2.25-0*\x)/1.5});
\draw [domain=0.0:4.499989389412392] plot(\x,{(-0--0.32*\x)/1.58});
\draw (0,1.5)-- (1.22,0.69);
\draw (0.98,0.2)-- (0,1.5);
\draw (0,1.5)-- (0.29,0.06);
\draw [shift={(0,0)}] plot[domain=0.2:1.57,variable=\t]({1*0.18*cos(\t r)+0*0.18*sin(\t r)},{0*0.18*cos(\t r)+1*0.18*sin(\t r)});
\draw [shift={(0,1.5)}] plot[domain=-1.35:0,variable=\t]({1*0.4*cos(\t r)+0*0.4*sin(\t r)},{0*0.4*cos(\t r)+1*0.4*sin(\t r)});

\draw [shift={(0,1.5)}] plot[domain=-0.59:0,variable=\t]({1*0.68*cos(\t r)+0*0.68*sin(\t r)},{0*0.68*cos(\t r)+1*0.68*sin(\t r)});
\draw (0.67,1.41) node[anchor=north west] {$ \alpha $};

\draw (1.0,0.02) node[anchor=north west] {$  \mathbf{S}_E$};
\draw (1.7,0.3) node[anchor=north west] {$  \mathbf{S}_X$};

\draw (0.4,1.5) node[anchor=north west] {$ \beta $};
\draw (0.07,0.3) node[anchor=north west] {$ \beta $};

\fill [color=black] (0,0) circle (1.5pt);
\draw[color=black] (-0.06,-.06) node {$0$};
\fill [color=black] (0,1.5) circle (1.5pt);
\draw[color=black] (0.0,1.6) node {$x$};
\fill [color=black] (-0.75,0.67) circle (1.5pt);
\draw[color=black] (-0.9,0.75) node {$\tau_2(x)$};
\fill [color=black] (0.75,0.67) circle (1.5pt);
\draw[color=black] (0.92,0.69) node {$\tau_1(x)$};
\fill [color=black] (0.93,1.5) circle (1.5pt);
\draw[color=black] (0.99,1.53) node {$v$};
\fill [color=black] (1.22,0.69) circle (1.5pt);
\draw[color=black] (1.33,0.72) node {$y$};
\fill [color=black] (0.98,0.2) circle (1.5pt);
\draw[color=black] (1.1,0.14) node {$u$};
\fill [color=black] (0.29,0.06) circle (1.5pt);
\draw[color=black] (0.45,0.0) node {$P_u x$};
\fill [color=black] (0.46,2.02) circle (1.5pt);
\draw[color=black] (0.53,1.99) node {$w$};
\end{tikzpicture}\caption{}
		\label{figure:angles}
	\end{figure}%
}
\customizepaperstyle{ams}{%
	\begin{figure}[H]
		\centering
\definecolor{uququq}{rgb}{0,0,0}
\begin{tikzpicture}[line cap=round,line join=round,x=3.0cm,y=3.0cm]
\clip(-1.1,-0.2) rectangle (2.7,2.1);
\fill[line width=0pt,fill=black,fill opacity=0.25] (1.74,3.44) -- (0,1.5) -- (0.75,0.67) -- (0.83,0.56) -- (0.89,0.46) -- (0.93,0.36) -- (0.98,0.2) -- (4.11,0.83) -- cycle;
\draw(0,0) circle (1);
\draw (0,1.5)-- (0,0);
\draw [dash pattern=on 2pt off 2pt,domain=-0.74535599249993:4.499989389412392] plot(\x,{(--1.12--0.83*\x)/0.75});
\draw [dash pattern=on 2pt off 2pt,domain=-3.982478398090284:0.7453559924999299] plot(\x,{(-1.12--0.83*\x)/-0.75});
\draw [rotate around={-35.33:(0,0)},dash pattern=on 2pt off 5pt] (0,0) ellipse (2.55 and 1.3);
\draw [domain=-3.98:4.5] plot(\x,{(--2.25-0*\x)/1.5});
\draw [domain=0.0:4.499989389412392] plot(\x,{(-0--0.32*\x)/1.58});
\draw (0,1.5)-- (1.22,0.69);
\draw (0.98,0.2)-- (0,1.5);
\draw (0,1.5)-- (0.29,0.06);
\draw [shift={(0,0)}] plot[domain=0.2:1.57,variable=\t]({1*0.18*cos(\t r)+0*0.18*sin(\t r)},{0*0.18*cos(\t r)+1*0.18*sin(\t r)});
\draw [shift={(0,1.5)}] plot[domain=-1.35:0,variable=\t]({1*0.4*cos(\t r)+0*0.4*sin(\t r)},{0*0.4*cos(\t r)+1*0.4*sin(\t r)});

\draw [shift={(0,1.5)}] plot[domain=-0.59:0,variable=\t]({1*0.68*cos(\t r)+0*0.68*sin(\t r)},{0*0.68*cos(\t r)+1*0.68*sin(\t r)});
\draw (0.67,1.41) node[anchor=north west] {$ \alpha $};

\draw (1.0,0.02) node[anchor=north west] {$  \mathbf{S}_E$};
\draw (1.7,0.3) node[anchor=north west] {$  \mathbf{S}_X$};

\draw (0.4,1.5) node[anchor=north west] {$ \beta $};
\draw (0.07,0.3) node[anchor=north west] {$ \beta $};

\fill [color=black] (0,0) circle (1.5pt);
\draw[color=black] (-0.06,-.06) node {$0$};
\fill [color=black] (0,1.5) circle (1.5pt);
\draw[color=black] (0.0,1.6) node {$x$};
\fill [color=black] (-0.75,0.67) circle (1.5pt);
\draw[color=black] (-0.9,0.75) node {$\tau_2(x)$};
\fill [color=black] (0.75,0.67) circle (1.5pt);
\draw[color=black] (0.92,0.69) node {$\tau_1(x)$};
\fill [color=black] (0.93,1.5) circle (1.5pt);
\draw[color=black] (0.99,1.53) node {$v$};
\fill [color=black] (1.22,0.69) circle (1.5pt);
\draw[color=black] (1.33,0.72) node {$y$};
\fill [color=black] (0.98,0.2) circle (1.5pt);
\draw[color=black] (1.1,0.14) node {$u$};
\fill [color=black] (0.29,0.06) circle (1.5pt);
\draw[color=black] (0.45,0.0) node {$P_u x$};
\fill [color=black] (0.46,2.02) circle (1.5pt);
\draw[color=black] (0.53,1.99) node {$w$};
\end{tikzpicture}\caption{}
		\label{figure:angles}
	\end{figure}%
}

Let $\beta:=\arccos(K^{-1})$ and $u\in\sphere E$ be such that $\duality{x^{\bot}}u>0$
and that the size of the angle formed by the rays $\ray 0x$ and $\ray 0u$
equals $\beta$ (i.e., $\duality xu=\norm x_{E}\cos\beta$). Let $\tau_{1}(x),\tau_{2}(x)\in\sphere E$
be the point(s) on the lines through $x$ that are tangent to $\sphere E$,
such that $\duality{x^{\bot}}{\tau_{1}(x)}\geq0$. Let
\[
w:=\begin{cases}
v & x=\tau_{1}(x)=\tau_{2}(x)\\
2x-\tau_{2}(x) & \mbox{otherwise,}
\end{cases}
\]
so that $w\in\ray{\tau_{2}(x)}x$ is distinct from $x$, and is such
that $\ray xw\subseteq\ray{\tau_{2}(x)}x$. Let $P_{u}$ denote the
orthogonal projection onto the span of $u$. Then size of the angle
formed between the rays $\ray x{P_{u}x}$ and $\ray xv$ is exactly
$\beta$. Since $\norm x_{E}\leq K$, the point $P_{u}x$ lies on
the line segment $\set{tu}{t\in(0,1]}$ (if $\norm x_{E}=K$, then
$P_{u}x=u=\tau_{1}(x)$), and therefore the size of angle between
rays $\ray xu$ and $\ray xv$ is at most $\beta$. Since $\ray x{\tau_{1}(x)}$
is between the rays $\ray xu$ and $\ray xv$, and since $\ray xv$
bisects the angle formed by the rays $\ray x{\tau_{1}(x)}$ and $\ray xw$,
the size of the angle formed by $r_{x,v}$ and $\ray xw$ is also
at most $\beta$.
Finally, by Lemma~\ref{lem:tangent-lines} (2),(3) and (4) and the fact that $\closedball E\subseteq\closedball X$, the ray $\ray xy$
lies either between the rays $\ray xu$ and $\ray xv$ or the rays
$\ray xv$ and $\ray xw$ (The point $y$ can only lie in the shaded
area in Figure~\ref{figure:angles}).
We conclude that $\alpha$,
the size of the angle between rays $\ray xv$ and $\ray xy$, is at
most $\beta=\arccos(K^{-1})$.\end{proof}
\begin{lem}
\label{lem:estimate_segment} Let $\norm{\cdot}_{E}$ denote the Euclidean
norm on $\R^{2}$ and $\norm{\cdot}_{X}$ be any norm on $\R^{2}$
such that $\closedball E\subseteq\closedball X\subseteq K\closedball E$
for some $K\geq1$. If $x,y\in\sphere X$ is such that $\theta$,
the size of angle between the rays $r_{0,x}$ and $r_{0,y}$, is at
most $\arccos(K^{-1})$, then
\[
\norm{x-y}_{E}\leq K^{2}\norm{\frac{x}{\norm x_{E}}-\frac{y}{\norm y_{E}}}_{E}.
\]
\end{lem}
\begin{proof}
As a visual aid, the reader is referred to Figure~\ref{figure:norm-bounf-figure}.

\customizepaperstyle{ams}{%
	\begin{figure}[H]
	\centering
	\begin{tikzpicture}[line cap=round,line join=round,x=2cm,y=2cm]
\clip(-1,-.5) rectangle (3.5,3.5);
\draw [shift={(0,0)}] plot[domain=0:3.141,variable=\t]({1*1*cos(\t r)+0*1*sin(\t r)},{0*1*cos(\t r)+1*1*sin(\t r)});
\draw [shift={(0,0)}] plot[domain=0:3.141,variable=\t]({1*3.5*cos(\t r)+0*3.5*sin(\t r)},{0*3.5*cos(\t r)+1*3.5*sin(\t r)});
\draw [rotate around={-45:(0,0)},dash pattern=on 2pt off 5pt] (0,0) ellipse (3.04 and 2.17);
\draw (0,2.5)-- (1.56,1.51);
\draw [domain=-4.82:6.87] plot(\x,{(-2.5-0*\x)/-1});
\draw (1.56,1.51)-- (0,0);
\draw (0,1)-- (1.03,1);
\draw (0,1.51)-- (1.56,1.51);
\draw (0,1)-- (0.72,0.7);
\draw (0.5,0.48)-- (0,1);
\draw (0,0)-- (0,2.5);
\draw [shift={(0,2.5)}] plot[domain=-0.56:0,variable=\t]({1*0.48*cos(\t r)+0*0.48*sin(\t r)},{0*0.48*cos(\t r)+1*0.48*sin(\t r)});
\draw [shift={(0,0)}] plot[domain=0.77:1.57,variable=\t]({1*0.31*cos(\t r)+0*0.31*sin(\t r)},{0*0.31*cos(\t r)+1*0.31*sin(\t r)});

\draw [shift={(1.56,1.51)}] plot[domain=2.58:3.141,variable=\t]({1*0.43*cos(\t r)+0*0.43*sin(\t r)},{0*0.43*cos(\t r)+1*0.43*sin(\t r)});


\draw (0.05,0.55) node[anchor=north west] {$\theta$};

\draw (1.0,0.22) node[anchor=north west] {$ \mathbf{S}_E $};
\draw (2.95,0.63) node[anchor=north west] {$ K \mathbf{S}_E $};
\draw (2.1,0.96) node[anchor=north west] {$ \mathbf{S}_X $};

\draw (0.45,2.5) node[anchor=north west] {$\alpha$};
\draw (0.90,1.75) node[anchor=north west] {$\alpha$};

\fill [color=black] (0,2.5) circle (1.5pt);
\draw[color=black] (0.13,2.6) node {$x$};
\fill [color=black] (1,2.5) circle (1.5pt);
\draw[color=black] (1.13,2.6) node {$v$};
\fill [color=black] (1.56,1.51) circle (1.5pt);
\draw[color=black] (1.74,1.54) node {$y$};
\fill [color=black] (0,0) circle (1.5pt);
\draw[color=black] (-0.05,-0.18) node {$0$};
\fill [color=black] (0,1) circle (1.5pt);
\draw[color=black] (-0.3,1.15) node {$x/\|x\|$};
\fill [color=black] (1.03,1) circle (1.5pt);
\draw[color=black] (1.25,1.0) node {$\lambda y$};
\fill [color=black] (0,1.51) circle (1.5pt);
\draw[color=black] (-0.23,1.48) node {$P_x y$};
\fill [color=black] (0.5,0.48) circle (1.5pt);
\draw[color=black] (0.6,0.35) node {$u$};
\fill [color=black] (0.72,0.7) circle (1.5pt);
\draw[color=black] (1.1,0.7) node {$y/\|y\|$};
\end{tikzpicture}\caption{}
	\label{figure:norm-bounf-figure}
	\end{figure}%
}
\customizepaperstyle{lms}{%
	\begin{figure}[b]
	\centering
	\begin{tikzpicture}[line cap=round,line join=round,x=2cm,y=2cm]
\clip(-1,-.5) rectangle (3.5,3.5);
\draw [shift={(0,0)}] plot[domain=0:3.141,variable=\t]({1*1*cos(\t r)+0*1*sin(\t r)},{0*1*cos(\t r)+1*1*sin(\t r)});
\draw [shift={(0,0)}] plot[domain=0:3.141,variable=\t]({1*3.5*cos(\t r)+0*3.5*sin(\t r)},{0*3.5*cos(\t r)+1*3.5*sin(\t r)});
\draw [rotate around={-45:(0,0)},dash pattern=on 2pt off 5pt] (0,0) ellipse (3.04 and 2.17);
\draw (0,2.5)-- (1.56,1.51);
\draw [domain=-4.82:6.87] plot(\x,{(-2.5-0*\x)/-1});
\draw (1.56,1.51)-- (0,0);
\draw (0,1)-- (1.03,1);
\draw (0,1.51)-- (1.56,1.51);
\draw (0,1)-- (0.72,0.7);
\draw (0.5,0.48)-- (0,1);
\draw (0,0)-- (0,2.5);
\draw [shift={(0,2.5)}] plot[domain=-0.56:0,variable=\t]({1*0.48*cos(\t r)+0*0.48*sin(\t r)},{0*0.48*cos(\t r)+1*0.48*sin(\t r)});
\draw [shift={(0,0)}] plot[domain=0.77:1.57,variable=\t]({1*0.31*cos(\t r)+0*0.31*sin(\t r)},{0*0.31*cos(\t r)+1*0.31*sin(\t r)});

\draw [shift={(1.56,1.51)}] plot[domain=2.58:3.141,variable=\t]({1*0.43*cos(\t r)+0*0.43*sin(\t r)},{0*0.43*cos(\t r)+1*0.43*sin(\t r)});


\draw (0.05,0.55) node[anchor=north west] {$\theta$};

\draw (1.0,0.22) node[anchor=north west] {$ \mathbf{S}_E $};
\draw (2.95,0.63) node[anchor=north west] {$ K \mathbf{S}_E $};
\draw (2.1,0.96) node[anchor=north west] {$ \mathbf{S}_X $};

\draw (0.45,2.5) node[anchor=north west] {$\alpha$};
\draw (0.90,1.75) node[anchor=north west] {$\alpha$};

\fill [color=black] (0,2.5) circle (1.5pt);
\draw[color=black] (0.13,2.6) node {$x$};
\fill [color=black] (1,2.5) circle (1.5pt);
\draw[color=black] (1.13,2.6) node {$v$};
\fill [color=black] (1.56,1.51) circle (1.5pt);
\draw[color=black] (1.74,1.54) node {$y$};
\fill [color=black] (0,0) circle (1.5pt);
\draw[color=black] (-0.05,-0.18) node {$0$};
\fill [color=black] (0,1) circle (1.5pt);
\draw[color=black] (-0.3,1.15) node {$x/\|x\|$};
\fill [color=black] (1.03,1) circle (1.5pt);
\draw[color=black] (1.25,1.0) node {$\lambda y$};
\fill [color=black] (0,1.51) circle (1.5pt);
\draw[color=black] (-0.23,1.48) node {$P_x y$};
\fill [color=black] (0.5,0.48) circle (1.5pt);
\draw[color=black] (0.6,0.35) node {$u$};
\fill [color=black] (0.72,0.7) circle (1.5pt);
\draw[color=black] (1.1,0.7) node {$y/\|y\|$};
\end{tikzpicture}\caption{}
	\label{figure:norm-bounf-figure}
	\end{figure}%
}

Let $x^{\bot}\in\sphere E\cap\{x\}^{\bot}$ be such that $\duality{x^{\bot}}y>0$
and define $v:=x+x^{\bot}.$ Let $P_{x}$ and $P_{y}$ be the orthogonal
projections onto the span of $x$ and $y$ respectively. Define $u:=P_{y}(x/\norm x_{E})$,
and $\lambda:=\norm{P_{x}y}_{E}^{-1}$ so that $P_{x}(\lambda y)=x/\norm x_{E}$.
Let $\alpha$ denote the size of the angle formed by the rays $\ray xy$
and $\ray xv$. We note that then size of the angle formed between
the rays $\ray{x/\norm x_{E}}{\lambda y}$ and $\ray{x/\norm x_{E}}u$
also equals~$\theta$. Elementary trigonometry will establish
\begin{eqnarray*}
\norm{x-y}_{E} & = & \frac{1}{\cos\alpha}\norm{y-P_{x}y}_{E}\\
 & = & \frac{1}{\lambda\cos\alpha}\norm{\lambda y-P_{x}(\lambda y)}_{E}\\
 & = & \frac{1}{\lambda\cos\alpha}\norm{\lambda y-\frac{x}{\norm x_{E}}}_{E}\\
 & = & \frac{1}{\lambda\cos\alpha\cos\theta}\norm{u-\frac{x}{\norm x_{E}}}_{E}.
\end{eqnarray*}

Now we note that $\norm{u-\frac{x}{\norm x}}_{E}\leq\norm{\frac{y}{\norm y_{E}}-\frac{x}{\norm x}}_{E}$,
since $u$ the is the closest point (with respect to $\norm{\cdot}_{E}$)
in the span of $y$ to the point $x/\norm x_{E}$. Also, by Lemma~\ref{cor:angles}
we have $\alpha\leq\arccos(K^{-1})$, so that $\cos\alpha\geq K^{-1}$.
Furthermore $\lambda^{-1}=\norm{P_{x}y}_{E}=\norm y_{E}\cos\theta\leq K\cos\theta$.
Finally we conclude\belowdisplayskip=-12pt
\begin{eqnarray*}
\norm{x-y}_{E} & = & \frac{1}{\lambda\cos\alpha\cos\theta}\norm{u-\frac{x}{\norm x_{E}}}_{E}\\
 & \leq & \frac{K\cos\theta}{\cos\alpha\cos\theta}\norm{\frac{y}{\norm y_{E}}-\frac{x}{\norm x}}_{E}\\
 & = & K^{2}\norm{\frac{y}{\norm y_{E}}-\frac{x}{\norm x}}_{E}.
\end{eqnarray*}\qedhere
\end{proof}

\begin{lem}
\label{lem:norm_decreasing} Let $\norm{\cdot}_{E}$ denote the Euclidean
norm on $\R^{2}$ and $\norm{\cdot}_{X}$ be any norm on $\R^{2}$
such that $\closedball E\subseteq\closedball X$. If $x,y\in\sphere X$,
then
\[
\norm{\frac{x}{\norm x_{E}}-\frac{y}{\norm y_{E}}}_{E}\leq\norm{x-y}_{E}.
\]
\end{lem}
\begin{proof}
As a visual aid we refer the reader to Figure~\ref{figure:euclidean-norm-bound-figure}.

\customizepaperstyle{ams}{%
	\begin{figure}[H]
	\centering
	\begin{tikzpicture}[line cap=round,line join=round,x=3.0cm,y=3.0cm]
\clip(-.5,-.2) rectangle (1.9,1.7);
\draw(0,0) circle (1);
\draw (0,0)-- (0,1.5);
\draw (0,0)-- (1,1.5);
\draw (0,1.5)-- (1,1.5);
\draw (0,1)-- (0.46,0.69);
\draw (0,1)-- (0.55,0.83);
\draw [rotate around={-146.12:(0.04,0.05)},dash pattern=on 2pt off 2pt] (0.04,0.05) ellipse (1.85 and 1.35);
\draw (0.67,1.01)-- (0,1);
\draw (1.01,0.3) node[anchor=north west] {$ \mathbf{S}_E $};
\draw (1.53,0.74) node[anchor=north west] {$ \mathbf{S}_X $};

\fill [color=black] (0,0) circle (1.5pt);
\draw[color=black] (-0.05,-0.07) node {$0$};
\fill [color=black] (0,1.5) circle (1.5pt);
\draw[color=black] (0.05,1.57) node {$x$};
\fill [color=black] (1,1.5) circle (1.5pt);
\draw[color=black] (1.07,1.55) node {$y$};
\fill [color=black] (0.55,0.83) circle (1.5pt);
\draw[color=black] (0.80,0.85) node {$y/\|y\|$};
\fill [color=black] (0,1) circle (1.5pt);
\draw[color=black] (-0.19,1.09) node {$x/\|x\|$};
\fill [color=black] (0.46,0.69) circle (1.5pt);
\draw[color=black] (0.55,0.65) node {$u$};
\fill [color=black] (0.67,1.01) circle (1.5pt);
\draw[color=black] (0.92,1.05) node {$y/\|x\|$};

\end{tikzpicture}\caption{}
	\label{figure:euclidean-norm-bound-figure}
	\end{figure}
}
\customizepaperstyle{lms}{%
	\begin{figure}[H]
	\centering
	\begin{tikzpicture}[line cap=round,line join=round,x=3.0cm,y=3.0cm]
\clip(-.5,-.2) rectangle (1.9,1.7);
\draw(0,0) circle (1);
\draw (0,0)-- (0,1.5);
\draw (0,0)-- (1,1.5);
\draw (0,1.5)-- (1,1.5);
\draw (0,1)-- (0.46,0.69);
\draw (0,1)-- (0.55,0.83);
\draw [rotate around={-146.12:(0.04,0.05)},dash pattern=on 2pt off 2pt] (0.04,0.05) ellipse (1.85 and 1.35);
\draw (0.67,1.01)-- (0,1);
\draw (1.01,0.3) node[anchor=north west] {$ \mathbf{S}_E $};
\draw (1.53,0.74) node[anchor=north west] {$ \mathbf{S}_X $};

\fill [color=black] (0,0) circle (1.5pt);
\draw[color=black] (-0.05,-0.07) node {$0$};
\fill [color=black] (0,1.5) circle (1.5pt);
\draw[color=black] (0.05,1.57) node {$x$};
\fill [color=black] (1,1.5) circle (1.5pt);
\draw[color=black] (1.07,1.55) node {$y$};
\fill [color=black] (0.55,0.83) circle (1.5pt);
\draw[color=black] (0.80,0.85) node {$y/\|y\|$};
\fill [color=black] (0,1) circle (1.5pt);
\draw[color=black] (-0.19,1.09) node {$x/\|x\|$};
\fill [color=black] (0.46,0.69) circle (1.5pt);
\draw[color=black] (0.55,0.65) node {$u$};
\fill [color=black] (0.67,1.01) circle (1.5pt);
\draw[color=black] (0.92,1.05) node {$y/\|x\|$};

\end{tikzpicture}\caption{}
	\label{figure:euclidean-norm-bound-figure}
	\end{figure}
}

Let $x,y\in\sphere X$. By exchanging the roles of $x$ and $y$ if
necessary, we may assume $\norm y_{E}\geq\norm x_{E}\geq1$. Let $P_{y}$
be the orthogonal projection onto the span of $y$ and let $u:=P_{y}\parenth{\frac{x}{\norm x_{E}}}$.
Then $\norm u_{E}\leq1$ and $\norm{y/\norm x_{E}}_{E}\geq1=\norm{y/\norm y_{E}}_{E}$.
Then, by the Pythagorean theorem,\belowdisplayskip=-12pt
\begin{eqnarray*}
\norm{\frac{x}{\norm x_{E}}-\frac{y}{\norm y_{E}}}_{E}^{2} & = & \norm{\frac{x}{\norm x_{E}}-u}_{E}^{2}+\norm{u-\frac{y}{\norm y_{E}}}_{E}^{2}\\
 & \leq & \norm{\frac{x}{\norm x_{E}}-u}_{E}^{2}+\norm{u-\frac{y}{\norm x_{E}}}_{E}^{2}\\
 & = & \norm{\frac{x}{\norm x_{E}}-\frac{y}{\norm x_{E}}}_{E}^{2}\\
 & = & \frac{1}{\norm x_{E}^{2}}\norm{x-y}_{E}^{2}\\
 & \leq & \norm{x-y}_{E}^{2}.
\end{eqnarray*}\qedhere
\end{proof}
In essence, the previous two Lemmas together establish that the local
bi\hyp{}Lipschitzness of the map $\sigma:\sphere X\to\sphere E$
defined by $\sigma(x):=x/\norm x_{E}$ when $\closedball E\subseteq\closedball X\subseteq K\closedball E$.

We will now use the previous results to prove one of our main results
which relates the intrinsic metric on $\sphere X$ to the induced
metric on $\sphere X$ when $\closedball E\subseteq\closedball X\subseteq K\closedball E$
for some $K\geq1$.
\begin{thm}
\label{thm:main-result1} Let $\norm{\cdot}_{E}$ denote the Euclidean
norm on $\R^{2}$ and $\norm{\cdot}_{X}$ be any norm on $\R^{2}$
such that $\closedball E\subseteq\closedball X\subseteq K\closedball E$
for some $K\geq1$. Then
\[
d_{X}\leq\pathmetric X{\sphere X}\leq K^{3}\frac{\pi}{2}d_{X}\quad\mbox{on}\quad\sphere X.
\]
\end{thm}
\begin{proof}
We have already noted in Section~\ref{sec:Preliminaries} that $d_{X}\leq\pathmetric X{\sphere X}$
on $\sphere X$.

Let $x,y\in\sphere X$ be arbitrary. Let $c:\R\to\R^{2}$ be the map
defined by $c(\theta):=(\cos(\theta),\sin(\theta))$ for $\theta\in\R$.
Let $\theta_{x},\theta_{y}\in\R$ be such that $c(\theta_{x})/\norm{c(\theta_{x})}_{X}=x$
and $c(\theta_{y})/\norm{c(\theta_{y})}_{X}=y$. By switching the
roles of $x$ and $y$, if necessary, we may assume that $0\leq\theta_{y}-\theta_{x}\leq\pi$.
Consider the path $\rho:[\theta_{x},\theta_{y}]\to\sphere X$ defined
by $\rho(\theta):=c(\theta)/\norm{c(\theta)}_{X}$ for $\theta\in[\theta_{x},\theta_{y}]$.

Let $\varepsilon>0$ be arbitrary and $\theta_{x}=\theta_{0}<\theta_{1}<\ldots<\theta_{n}=\theta_{y}$
be a partition of $[\theta_{x},\theta_{y}]$ such that
\[
\sum_{j=0}^{n-1}\norm{\rho(\theta_{j})-\rho(\theta_{j+1})}_{X}\geq\lengthoperator X(\rho)-\varepsilon.
\]
We may assume that $0<\theta_{j+1}-\theta_{j}\leq\arccos K^{-1}$,
since the triangle inequality ensures that every refinement of $\curly{\theta_{j}}_{j=1}^{n}$
still satisfies the above inequality.

We note that $\closedball E\subseteq\closedball X\subseteq K\closedball E$
implies $\norm w_{X}\leq\norm w_{E}\leq K\norm w_{X}$ for all $w\in\R^{2}$.
Then, by Lemmas~\ref{lem:estimate_segment},~ \ref{lem:bound-eclidean-arclength-by-shortcut}~and~\ref{lem:norm_decreasing},
we obtain
\customizepaperstyle{ams}{%
	\begin{eqnarray*}
	 &  & \negthickspace\negthickspace\negthickspace\negthickspace\negthickspace\negthickspace\negthickspace\negthickspace\negthickspace\negthickspace\negthickspace\negthickspace\negthickspace\negthickspace \pathmetric X{\sphere X}(x,y)\\
	 & \leq & \lengthoperator X(\rho)\\
	 & \leq & \sum_{j=0}^{n-1}\norm{\rho(\theta_{j})-\rho(\theta_{j+1})}_{X}+\varepsilon\\
	 & \leq & \sum_{j=0}^{n-1}\norm{\rho(\theta_{j})-\rho(\theta_{j+1})}_{E}+\varepsilon\\
	 & \leq & \sum_{j=0}^{n-1}K^{2}\norm{c(\theta_{j})-c(\theta_{j+1})}_{E}+\varepsilon\\
	 & \leq & K^{2}\sup\set{\sum_{j=0}^{m-1}\norm{c(\phi_{j})-c(\phi_{j+1})}_{E}}{
	 			\begin{array}{c}
		 			m\in\N \\
		 			\theta_{x}=\phi_{0}<\ldots<\phi_{m}=\theta_{y}
	 			\end{array}
	 		}+\varepsilon\\
	 & \leq & K^{2}\pathmetric E{\sphere E}\parenth{\frac{x}{\norm x_{E}},\frac{y}{\norm y_{E}}}+\varepsilon\\
	 & \leq & K^{2}\frac{\pi}{2}\norm{\frac{x}{\norm x_{E}}-\frac{y}{\norm y_{E}}}_{E}+\varepsilon\\
	 & \leq & K^{2}\frac{\pi}{2}\norm{x-y}_{E}+\varepsilon\\
	 & \leq & K^{3}\frac{\pi}{2}\norm{x-y}_{X}+\varepsilon.
	\end{eqnarray*}%
}
\customizepaperstyle{lms}{
	\begin{eqnarray*}
			 \pathmetric X{\sphere X}(x,y) & \leq & \lengthoperator X(\rho)\\
					 & \leq & \sum_{j=0}^{n-1}\norm{\rho(\theta_{j})-\rho(\theta_{j+1})}_{X}+\varepsilon\\
					 & \leq & \sum_{j=0}^{n-1}\norm{\rho(\theta_{j})-\rho(\theta_{j+1})}_{E}+\varepsilon\\
					 & \leq & \sum_{j=0}^{n-1}K^{2}\norm{c(\theta_{j})-c(\theta_{j+1})}_{E}+\varepsilon\\
					 & \leq & K^{2}\sup\set{\sum_{j=0}^{m-1}\norm{c(\phi_{j})-c(\phi_{j+1})}_{E}}{
					 		\begin{array}{c}
					 			m\in\N \\
				 				\theta_{x}=\phi_{0}<\ldots<\phi_{m}=\theta_{y}
				 			\end{array}
					 		}+\varepsilon\\
					 & \leq & K^{2}\pathmetric E{\sphere E}\parenth{\frac{x}{\norm x_{E}},\frac{y}{\norm y_{E}}}+\varepsilon\\
					 & \leq & K^{2}\frac{\pi}{2}\norm{\frac{x}{\norm x_{E}}-\frac{y}{\norm y_{E}}}_{E}+\varepsilon\\
					 & \leq & K^{2}\frac{\pi}{2}\norm{x-y}_{E}+\varepsilon\\
					 & \leq & K^{3}\frac{\pi}{2}\norm{x-y}_{X}+\varepsilon.
	\end{eqnarray*}%
}
Since $\varepsilon>0$ was chosen arbitrarily, the result follows.
\end{proof}
Our final result now follows through an easy application of the previous
result and John's Theorem (Theorem~\ref{thm:John's-Theorem}):
\begin{thm}
\label{thm:main-result2} For any norm $\norm{\cdot}_{X}$ on a real
vector space $V$,
\[
d_{X}\leq\pathmetric X{\sphere X}\leq\sqrt{2}\pi\,d_{X}\quad\mbox{on}\quad\sphere X.
\]
\end{thm}
\begin{proof}
We have already noted in Section~\ref{sec:Preliminaries} that $d_{X}\leq\pathmetric X{\sphere X}\leq\planarpathmetric X{\sphere X}$
on $\sphere X$. Let $x,y\in\sphere X$ be arbitrary and let $W\subseteq V$
be any two-dimensional subspace containing $x$ and $y$, noting that
then $\planarpathmetric X{\sphere X}(x,y)\leq\pathmetric X{\sphere X\cap W}(x,y).$
By John's Theorem (Theorem~\ref{thm:John's-Theorem}), there exists
a Euclidean norm $\norm{\cdot}_{E}$ on $W$ such that $\norm w_{X}\leq\norm w_{E}\leq\sqrt{2}\norm w_{X}$
for all $w\in W$, i.e., $\closedball E\subseteq\closedball X\cap W\subseteq\sqrt{2}\closedball E$.
Then, by Theorem~\ref{thm:main-result1}, we may conclude that $d_{X}\leq\pathmetric X{\sphere X}\leq\sqrt{2}\pi\,d_{X}$
on $\sphere X$.
\end{proof}

\bibliographystyle{amsplain}
\bibliography{bib}

\end{document}